\documentclass[10pt]{amsart}

\usepackage{amsfonts, amsmath, amsthm, amssymb, epsfig, bm, graphics, psfrag, latexsym, chngcntr, color, yfonts}
\usepackage[all,cmtip]{xy}
\usepackage{tikz} \usetikzlibrary{matrix,arrows}
\definecolor{darkblue}{rgb}{0,0,0.4}
\usepackage[colorlinks=true, citecolor=darkblue, filecolor=darkblue, linkcolor=darkblue, urlcolor=darkblue]{hyperref}
\usepackage[all]{hypcap}

\newtheorem{thm}{Theorem}[section]         
\newtheorem{lem}[thm]{Lemma}

\newtheorem{defn}[thm]{Definition} 
  
\counterwithin{figure}{section}

\newcommand{\R}{\mathbb{R}}
\newcommand{\Z}{\mathbb{Z}}
\newcommand{\F}{\mathbb{F}}

\newcommand{\mc}{\mathcal}
\newcommand{\mb}{\mathbb}
\newcommand{\mf}{\mathfrak}

\newcommand{\wt}{\widetilde}

\newcommand{\del}{\partial}
\newcommand{\sbs}{\subset}

\newcommand{\sm}{\setminus}

\newcommand{\al}{\alpha}
\newcommand{\be}{\beta}

\newcommand{\Zoltan}{Zolt\'{a}n}
\newcommand{\Szabo}{Szab\'{o}}
\newcommand{\Ozsvath}{Ozsv\'{a}th}
\newcommand{\Andras}{Andr\'{a}s}
\newcommand{\Juhasz}{Juh\'{a}sz}

\newcommand{\ith}{^{\text{th}}}

\newcommand{\Mor}{\operatorname{Mor}}
\newcommand{\Ob}{\operatorname{Ob}}

\newcommand{\HFK}{\mathit{HFK}}

\begin{document}

\title{Grid diagrams and shellability} 
\author{Sucharit Sarkar} 
\address{Department of Mathematics, Columbia University, New York, NY 10027} 
\email{\href{mailto:sucharit.sarkar@gmail.com}{sucharit.sarkar@gmail.com}} 

\keywords{knot Floer homology; shellable poset; grid diagram; flow category} 
\subjclass[2010]{\href{http://www.ams.org/mathscinet/search/mscdoc.html?code=57M25,06A07,57R58}{57M25,     06A07, 57R58}}

\date{}

\begin{abstract}
  We explore a somewhat unexpected connection between knot Floer   homology and shellable posets, via grid diagrams. Given a grid   presentation of a knot $K$ inside $S^3$, we define a poset 
which has   an associated chain complex whose homology is the knot Floer   homology of $K$. We then prove that the closed intervals of this   poset are shellable. This allows us to 
combinatorially associate a PL flow category to a grid diagram. \end{abstract}

\maketitle

\section{Introduction}

Heegaard Floer homology is a powerful invariant for closed oriented $3$-manifolds, introduced by Peter \Ozsvath{} and \Zoltan{} \Szabo{} \cite{POZSz,POZSzapplications}. This invariant was later generalized by them \cite{POZSzknotinvariants} and independently by Jacob Rasmussen \cite{JR} to an invariant called knot Floer homology for knots in $3$-manifolds, which was later further generalized to include the case of links \cite{POZSzlinkinvariants}.  We will mostly be concerned with $\mathit{HFK}^{-}(S^3,K;\F_2)$, the minus version of knot Floer homology of a knot $K\sbs S^3$ with coefficients in $\F_2$. There are two gradings $M$ and $A$ on $\HFK^{-}(S^3,K;\F_2)$ called the Maslov grading and the Alexander grading. The $\F_2$-module $\mathit{HFK}^{-}(S^3,K;\F_2)$ is obtained as the homology of a certain chain complex, and the Maslov grading is in fact the homological grading.

The strength of knot Floer homology can be demonstrated by the following few theorems. Peter \Ozsvath{} and \Zoltan{} \Szabo{} proved that a version of knot Floer homology detects the $3$-ball genus of the knot \cite{POZSzgenusbounds}, and a version of the link Floer homology describes the Thurston polytope of the link \cite{POZSzthurstonnorm}. Yi Ni showed that knot Floer homology can also be used to determine if a knot is fibered \cite{YN}. Peter \Ozsvath{} and \Zoltan{} \Szabo{} also constructed an invariant $\tau$ coming from knot Floer homology which gives a lower bound on the $4$-ball genus of a knot \cite{POZSz4ballgenus}.

Based on a grid presentation of a knot $K\sbs S^3$, \cite{CMPOSS} constructs a chain complex over $\mathbb{F}_2$, whose homology agrees with $\mathit{HFK}^{-}(S^3,K;\F_2)$. A grid diagrams is a way of representing a knot $K$ in $S^3$. Grid diagrams were first introduced as arc-presentations in \cite{gridHB}, and they are equivalent to the square-bridge positions of \cite{gridHCL}, the Lengendrian realisations of \cite{gridHM}, the asterisk presentations of \cite{gridLN} and the fences of \cite{gridLR}. Grid diagrams have been explored in great detail in \cite{gridPC}, and the chain complexes coming from grid diagrams have been studied extensively in \cite{CMPOZSzDT}.

In this short paper, we will explore an unexpected connection between knot Floer homology, grid diagrams, shellability and flow categories. Shellability is a fairly strong condition for partially ordered sets, see \cite{shellAB, shellABcwcomplexes, shellABMW}. We will prove that the grid chain complex comes from a poset whose every closed interval is shellable.
Shellable posets carry many rich combinatorial structures, and the author hopes that some of these structures might lead to a better understanding of knot Floer homology and to new knot invariants.

In particular, shellability forces many geometric constraints on the order complex of the poset; namely, the order complex of a thin (resp. subthin) shellable poset is a sphere (resp. a ball). This allows us to construct a flow category, in the sense of \cite{generalRCJJGS}, from a grid diagram. \cite{generalRCJJGS} gives a recipe for constructing a stable homotopy type, starting from a flow category and a coherent choice of framings of the tangent bundles of the various manifolds that appear. Thus, in order to construct a stable homotopy type from a grid presentation of a knot, we need to choose coherent framings of the tangent bundles. Therefore, it will be interesting to investigate what additional combinatorial properties we need on the posets in order to construct such coherent framings, and whether the posets coming from grid diagrams satisfy those propeties.

\subsection*{Acknowledgment}
A part of the work was done when the author was supported by Princeton Honorific Fellowship and partially supported by the Princeton Centennial Fellowship, and a part of the work was done when he was fully supported by the Clay Research Fellowship. He is grateful to Boris Bukh and Sarah Rasmussen for introducing him to shellable posets. He would also like to thank Chris Douglas, \Andras{} \Juhasz{}, Robert Lipshitz, Ciprian Manolescu, Peter \Ozsvath{}, Jacob Rasmussen and Dylan Thurston for some interesting suggestions and many helpful discussions. Finally, he would like to thank \Zoltan{} \Szabo{} for introducing him to the fascinating world of Heegaard Floer homology in the first place.

\section{Partially ordered sets}\label{sec:posets}

In this section we give a brief overview of some well-known concepts related to partially ordered sets and shellability, following \cite{shellAB}.  A set $P$ with a binary relation $\preceq$ is a \emph{partially ordered set} if $a\preceq b,b\preceq c\Rightarrow a\preceq c$ and $a\preceq b,b\preceq a\Leftrightarrow a=b$. If $a\preceq b, a\neq b$, then we often say that $a$ is less than $b$ and write $a\prec b$.  We also often abbreviate partially ordered sets as \emph{posets}.  A poset $P$ gives rise to the small category $\mc{C}(P)$, whose objects are the elements of $P$, and the set of morphisms from $x$ to $y$ is non-empty if and only if $y\preceq x$, and in that case there is a unique morphism.

We say that $b$ \emph{covers} $a$, and write $a\leftarrow b$ if $a\prec b$ and $\nexists z, a\prec z\prec b$. If $\nexists z, b\prec z$, then we say that $b$ is a \emph{maximal element}. \emph{Minimal   elements} are defined similarly. If $b$ is covered by a maximal element, we say that $b$ is a \emph{submaximal element}.

Any subset of a poset has an induced partial order. A subset $C\subseteq P$ is called a \emph{chain} if the induced order on $C$ is a total order. Given a poset $P$, we can create another poset $B(P)$, called the \emph{barycentric subdivision} of $P$, whose elements are the chains of $P$, partially ordered by inclusion. \emph{Maximal   chains} and \emph{submaximal chains} of $P$ are the maximal elements and submaximal elements of $B(P)$ respectively. The \emph{length} of a chain is the cardinality of the chain considered just as a set.

The \emph{Cartesian product} of two posets $P$ and $Q$ is defined as the poset $P\times Q$, whose elements are pairs $(p,q)$ with $p\in P$ and $q\in Q$, and we declare $(p',q')\preceq (p,q)$ if and only if $p'\preceq p$ in $P$ and $q'\preceq q$ in $Q$. The \emph{order   complex} of a poset $P$ is the simplicial complex $X(P)$, whose $k$-simplices are chains of length $(k+1)$. The boundary maps are defined naturally. For any finite poset $P$, the simplicial complex $X(B(P))$ is the (first) barycentric subdivision of the simplicial complex $X(P)$. For any two finite posets $P$, $Q$, the space $X(P)\times X(Q)$ is naturally homeomorphic to the simplicial complex $X(P\times Q)$.

We define a \emph{closed interval} $[a,b]$ as $\{z\in P\mid a\preceq z\preceq b\}$. The other types of intervals are $(a,b)=\{z\in P\mid a\prec z\prec b\}$, $[a,b)=\{z\in P\mid a\preceq z\prec b\}$, $(a,b]=\{z\in P\mid a\prec z\preceq b\}$, $(-\infty,b]=\{z\in P\mid z\preceq b\}$, $(-\infty,b)=\{z\in P\mid z\prec b\}$, $[a,\infty)=\{z\in P\mid a\preceq z\}$, $(a,\infty)=\{z\in P\mid a\prec z\}$ and $(-\infty,\infty)=P$. A poset is said to be \emph{graded} if in every interval, all the maximal chains have the same length, in which case the common length is known as the length of the interval. A graded poset is said to be \emph{thin}, if every submaximal chain is covered by exactly $2$ maximal chains.  A graded poset is \emph{subthin} if it is not thin, and every submaximal chain is covered by at most $2$ maximal chains.

A graded poset is said to be \emph{shellable}\phantomsection\label{shellable} if the maximal chains have a total ordering $\leq$, such that $\mathfrak{m}_i < \mathfrak{m}_j\Rightarrow \exists \mathfrak{m}_k < \mathfrak{m}_j$ and $\exists x\in\mathfrak{m}_j$ such that $\mathfrak{m}_i\cap\mathfrak{m}_j\subseteq\mathfrak{m}_k \cap\mathfrak{m}_j=\mathfrak{m}_j\setminus\{x\}$. A more geometric way of saying this is the following. A graded poset $P$ is shellable if the maximal dimensional simplices of its order complex $X(P)$ can be totally ordered in some way, such that each such simplex intersects the union of the smaller such simplices in a non-empty union of maximal dimensional faces on its boundary.

\begin{thm} \cite{shellAB} \label{thm:shellbasic}
If a finite poset $P$ is shellable, then every interval of $P$ is shellable. Furthermore, the barycentric subdivision $B(P)$ is also shellable.
\end{thm}

A graded poset is said to be edge-lexicographically shellable or \emph{EL-shellable}\phantomsection\label{el} if there is a map $f$ from the set of all covering relations  to a totally ordered set, such that for any closed interval $[x_1,x_n]$ of length $n$, if we associate the $(n-1)$-tuple \emph{labeling} $(f([x_1,x_2]),\ldots,f([x_{n-1},x_n]))$ to a maximal chain $\{x_1\leftarrow x_2\cdots \leftarrow x_{n-1}\leftarrow x_n\}$, then there is a unique maximal chain for which the $(n-1)$-tuple labeling is \emph{increasing}, and under the lexicographic ordering, the corresponding $(n-1)$-tuple labeling is the \emph{smallest} one among the labelings coming from maximal chains between $x_1$ and $x_n$.

\begin{thm}\cite{shellAB} \label{thm:el}
If a finite poset $P$ is EL-shellable, then every closed interval of $P$ is shellable.
\end{thm}

The following theorem explores the geometric properties of shellable posets and suggests that shellablility is a strong condition.

\begin{thm} \cite{shellGDVK} \label{thm:ordercomplex}
The order complex of a finite, shellable   and thin poset of length $n+1$ is PL-homeomorphic to the $n$-dimensional sphere. The order complex of   a finite, shellable and subthin poset of length $n+1$ is PL-homeomorphic to the $n$-dimensional ball,   and the boundary of the ball corresponds to those submaximal chains   which are covered by exactly one maximal chain.
\end{thm}

In \hyperref[sec:griddiagrams]{Section \ref*{sec:griddiagrams}}, we will encounter posets with the following properties. A \emph{grading assignment} is a map $g$ from the elements of the poset to $\mathbb{Z}$, such that whenever $a\leftarrow b$, $g(b)=g(a)+1$. Having a grading assignment is weaker than being graded, but is stronger than each closed interval being graded. A poset, where every closed interval is graded, is said to be \emph{locally thin} if every closed interval of length $3$ has exactly two maximal chains. This is equivalent to saying that every interval of the form $(y,x)$ is thin. A \emph{GT poset}\phantomsection\label{gtposet} is a locally thin poset equipped with a grading assignment, such that there are only finitely many elements in each grading.

\begin{defn}\label{defn:chaingt} Given a GT poset $P$, we can associate to it a chain complex $C(P)$ over $\F_2$, defined as follows. The $i\ith$ chain group $C_i$ is the $\F_2$-module freely generated by the elements of $P$ with grading $i$. The boundary map $\del_i:C_i\rightarrow C_{i-1}$ is defined as $\del x=\sum_{y\leftarrow x}y$. 
\end{defn}

\section{Grid diagrams}\label{sec:griddiagrams}

In this section we will introduce grid diagrams and associate certain posets to them. A grid diagram is a picture on the standard torus, although for convenience, we often think of it as a diagram on a square in the plane.  Much of the material in this section comes from \cite{CMPOSS,CMPOZSzDT}. The interested reader should consult \cite{CMPOZSzDT} for a more complete description of grid diagrams.

A \emph{grid diagram of index $n$} is a picture on the standard torus $T$. There are $n$ $\alpha$ circles, which are pairwise disjoint and parallel to the meridian, such that they cut up the torus into $n$ horizontal annuli, and there are $n$ $\beta$ circles, which are pairwise disjoint and parallel to the longitude, such that they cut up the torus into $n$ vertical annuli. Furthermore, each $\alpha$ circle intersects with each $\beta$ circle exactly once, so clearly $T\setminus (\alpha\cup\beta)$ has $n^2$ components. There are $2n$ markings on $T\setminus (\alpha\cup\beta)$, numbered $X_1,\ldots,X_n, O_1,\ldots,O_n$, such that each horizontal annulus contains $X_i$ and $O_i$ for some $i$, and each vertical annulus contains $O_i$ and $X_{i+1}$, for some $i$, with the numbering being done modulo $n$. \hyperref[fig:trefoil]{Figure \ref*{fig:trefoil}} shows a grid diagram of index $5$.

\begin{figure} 
\psfrag{x1}{$X_1$}
\psfrag{x2}{$X_2$}
\psfrag{x3}{$X_3$}
\psfrag{x4}{$X_4$}
\psfrag{x5}{$X_5$}
\psfrag{o1}{$O_1$}
\psfrag{o2}{$O_2$}
\psfrag{o3}{$O_3$}
\psfrag{o4}{$O_4$}
\psfrag{o5}{$O_5$}
\begin{center}
\includegraphics[width=170pt]{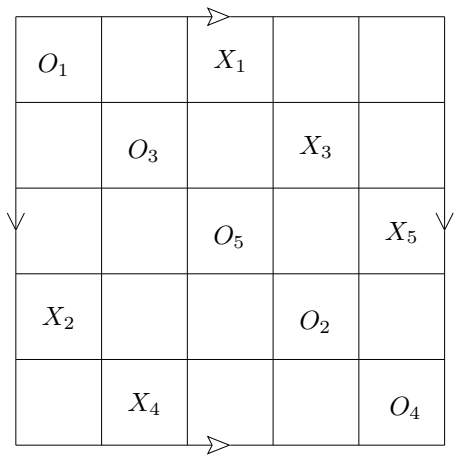}
\end{center}
\caption{A grid diagram for the trefoil.}\label{fig:trefoil}
\end{figure}

Given a grid diagram, we can construct a knot inside $\mb{R}^3$ as follows. If $T$ is embedded in $\mathbb{R}^3$ in the standard way, with the meridian bounding a disk inside the torus, and the longitude bounding a disk outside, then the knot is obtained by joining $X_i$ to $O_i$ in the horizontal annuli inside the torus $T$, and by joining $O_i$ to $X_{i+1}$ in the vertical annuli outside the torus $T$.
For example, the grid diagram of \hyperref[fig:trefoil]{Figure \ref*{fig:trefoil}} represents the trefoil. In the other direction, given a knot $K\subset \mathbb{R}^3$, it is not difficult to get a grid diagram for $K$.

\begin{lem}\cite{gridPC} For every knot $K\subset\mathbb{R}^3$, there is   a grid diagram that represents $K$.
\end{lem}

Given a grid diagram of index $n$ representing a knot $K$, we will define a GT poset $\mathcal{G}$, such that the homology of the associated chain complex is isomorphic to minus version of knot Floer homology over $\F_2$.

A \emph{generator} is a formal sums $\wt{x}=x_1+x_2+\cdots+x_n$ of $n$ points, such that each $\alpha$ circle contains one point and each $\beta$ circle contains one point. The set of all the $n!$ generators is denoted by $\wt{\mc{G}}$. The set $\mathcal{G}$ consists of elements of the form $x=\wt{x} \prod_{i=1}^n U_i^{k_i}$ where $\wt{x}\in\wt{\mathcal{G}}$ and $k_i\in\mathbb{N}\cup\{0\}$.  
We need the following few definitions to understand the partial order on $\mc{G}$.

A \emph{domain} $D$ connecting a generator $\wt{x}$ to another generator $\wt{y}$, is a $2$-chain in $T\setminus(\alpha\cup\beta)$ (i.e. a linear combination of the component of $T\sm(\al\cup\be)$) such that $\partial(\partial D|_{\alpha})=\wt{y}-\wt{x}$. Every domain $D$ can be associated to an integer valued index called the \emph{Maslov index} $\mu(D)$. The set of all domains connecting $\wt{x}$ to $\wt{y}$ is denoted by $\mathcal{D}(\wt{x},\wt{y})$. For a point $p\in T\setminus(\alpha\cup\beta)$ and a $2$-chain $D$, we define $n_p(D)$ to be the coefficient of $D$ at the point $p$. We define $\mathcal{D}^0(\wt{x},\wt{y})$ as a subset of $\mathcal{D}(\wt{x},\wt{y})$ consisting of all the domains $D$ with $n_p(D)=0$ whenever $p$ is one of the $2n$ $X$ or $O$ markings.  If $x=\wt{x}\prod_i U_i^{k_i}$ and $y=\wt{y} \prod_i U_i^{l_i}$ are two elements in $\mathcal{G}$, we define $\mathcal{D}(x,y)$ as the subset of $\mathcal{D}(\wt{x},\wt{y})$ consisting of all the domains $D$ with $n_{O_i}(D)=l_i-k_i$ and $n_{X_i}(D)=0$.  A $2$-chain $D$ is positive if $n_p(D)\geq 0$ for all points $p\in T\setminus(\alpha\cup\beta)$.

\begin{lem}\cite[Definition 3.4]{POZSzlinkinvariants} For any $\wt{x}\in\wt{\mc{G}}$, the set $\mathcal{D}^0(\wt{x},\wt{x})$ consists of only the   trivial domain.  Therefore, for any pair $x,y\in\mc{G}^{-}$, the set $\mathcal{D}(x,y)$ has at most one element.
\end{lem}

We choose an $\alpha$ circle and a $\beta$ circle on the grid diagram $G$ and cut open the torus $T$ along those circles to obtain a diagram in $[0,N)\times[0,N)\subset \mathbb{R}^2$. In this planar diagram, the $\alpha$ circles become the lines $y=i$ and the $\beta$ circles become the lines $x=i$ for $0\le i<N$.
Following \cite{CMPOZSzDT}, for two points $a=(a_1, a_2)$ and $b=(b_1,b_2)$ in $\mathbb{R}^2$, we define $J(a,b)=\frac{1}{2}$ if $(a_1-b_1)(a_2-b_2)>0$ and $0$ otherwise. We extend $J$ bilinearly for linear combinations of points. Let $\mathbb{O}$ and $\mb{X}$ denote the formal sums $\sum_i O_i$ and $\sum_i X_i$, respectively. For $\wt{x}\in\wt{\mathcal{G}}$, we define the \emph{Maslov   grading} $M(\wt{x})=J(\wt{x}- \mathbb{O},\wt{x}-\mathbb{O})+1$ and the \emph{Alexander grading} $A(\wt{x})=J(\wt{x}- \frac{\mathbb{X}+\mathbb{O}}{2},\mathbb{X}-\mathbb{O})-\frac{N-1}{2}$. The following is straightforward.

\begin{lem}\cite[Section 2.2]{CMPOZSzDT}
  $A(\wt{x})$ and $M(\wt{x})$ are integer valued gradings which are independent of the choice   of $\alpha$ and $\beta$ circles along which the torus is cut open. If $D\in\mc{D}(\wt{x},\wt{y})$ is a domain, then $M(\wt{x})-M(\wt{y})=\mu(D)-2\sum_i n_{O_i}(D)$ and $A(\wt{x})-A(\wt{y})=\sum_i(n_{X_i}(D)-n_{O_i}(D))$.
\end{lem}

We extend the assignment of Maslov and Alexander gradings from $\wt{\mathcal{G}}$ to $\mathcal{G}$ by defining $M(\wt{x} \prod_i U_i^{k_i})=M(\wt{x})-2\sum_i k_i$ and $A(\wt{x} \prod_i U_i^{k_i})=A(\wt{x})-\sum_i k_i$. Stated differently, we assign an $(M,A)$ bigrading of $(-2,-1)$ to each $U_i$. 

If the reader is following the analogies from the Floer homology picture, it should be pretty clear by this point that the positive domains of Maslov index one are of special importance to us. \hyperref[lem:maslovone]{Lemma \ref*{lem:maslovone}} characterizes them. Note that the lemma also follows from \hyperref[lem:grading]{Lemma \ref*{lem:grading}}.

A domain $R\in\mc{D}(\wt{x},\wt{y})$ is called an \emph{empty   rectangle}\phantomsection\label{emptyrectangle} if $R$ has coefficients $0$ and $1$ everywhere, and the closure of the region where $R$ has coefficient $1$ forms a rectangle which does not contain any $\wt{x}$-coordinate or any $\wt{y}$-coordinate in its interior. It is clear that empty rectangles have Maslov index one \cite[Equation 12]{CMPOSS}. The set of empty rectangles joining $\wt{x}$ to $\wt{y}$ is denoted by $\mathcal{R}(\wt{x},\wt{y})$. Note that
$\mathcal{R}(\wt{x},\wt{y})=\varnothing$ unless
$\wt{x}$ and $\wt{y}$ differ in exactly two coordinates, and
even then $\#|\mathcal{R}(\wt{x},\wt{y})|\leq 2$.  
For $x=\wt{x}\prod_i U_i^{k_i}$ and $y=\wt{y}\prod_i
U_i^{l_i}$ in $\mathcal{G}$, we define
$\mathcal{R}(x,y)=\mathcal{R}(\wt{x},\wt{y})\cap\mathcal{D}(x,y)$.

\begin{lem}\cite{CMPOSS}\label{lem:maslovone} If $D\in\mathcal{D}(\wt{x},\wt{y})$ is a positive domain with
  $\mu(D)=1$,  then $D$ is an empty rectangle.
\end{lem}

\begin{lem}\label{lem:grading} Let $D\in\mathcal{D}(\wt{x},\wt{y})$ be   a positive domain.  Then there exist generators   $\wt{u}_0,\wt{u}_1,\cdots,\wt{u}_k\in\wt{\mathcal{G}}$   with $\wt{u}_0=\wt{x}$ and $\wt{u}_k=\wt{y}$,   and empty rectangles $D_i\in\mathcal{R}(\wt{u}_{i-1},\wt{u}_i)$   such that $D=\sum_i D_i$.
\end{lem}

\begin{proof} We can assume that $D$ is not a trivial domain, and   thereby assume without loss of generality that $n_{\wt{x}_1}(D)\ne 0$.   Furthermore, since $\partial (\partial   D|_{\alpha})=\wt{y}-\wt{x}$, the coefficient of $D$ at   either the top-right square or the bottom-left square of $\wt{x}_1$ must   be non-zero. Assume, after rotating everything by $180^{\circ}$ if necessary, that the coefficient of the top-right square is non-zero.

Consider all rectangles $R$, such that $R$ is contained in $D$ as $2$-chains (i.e. the $2$-chain $D\sm R$ is positive), and $R$ has $\wt{x}_1$ as its bottom-left corner. Partially order such rectangles by inclusion. Let $R_0$ be a maximal element under such an order, and let $p_0$ be the top-right corner of $R_0$. We want to show that $R_0$ contains an $\wt{x}$-coordinate other than $\wt{x}_1$. 

Assume that $D$ has non-zero coefficient at the square to the   top-left of $p_0$. Since $R_0$ is a maximal element, either $p_0$ must lie on the $\alpha$ circle immediately below the $\alpha$ passing through $\wt{x}_1$, or $D$ must have   zero coefficient at some square above the top edge of   $R$. In the first case, $R_0$ contains the $\wt{x}$-coordinate lying on the $\beta$ circle passing through $p_0$, and so we are done. For the second case, let us start at $p_0$ and proceed left along the top edge of $R_0$ until we reach the first point $p_1$, such $D$ has non-zero   coefficient at the top-right square of $p_1$, but has zero   coefficient at the top-left square of $p_1$. Then it is easy to see   that $p_1$ must be an $\wt{x}$-coordinate, and once more, we are done. A similar analysis shows that if $D$ has non-zero   coefficient at the bottom-right square of $p_0$, then also $R_0$   contains an $\wt{x}$-coordinate other than $\wt{x}_1$.  Finally, if the   coefficient of $D$ is zero at both the top-left and the bottom-right   square of $p_0$, then $p_0$ itself is an $\wt{x}$-coordinate.

Thus $D$ contains a rectangle $R_1$, with two $\wt{x}$-coordinates, say $\wt{x}_1$ and $\wt{x}_2$, being the bottom-left corner and the top-right corner respectively.  Now consider the partial order on rectangles that we have defined earlier, but restrict only to the ones whose top-right corner is an $\wt{x}$-coordinate.  Let $R_3$ be a minimal element. Then the rectangle $R_3$ is an empty rectangle connecting $\wt{x}$ to some generator $\wt{u}_1$. The positive domain $D\sm R_3\in\mc{D}(\wt{u}_1,\wt{y})$ has a smaller sum of coefficients as $2$-chains, and hence an induction finishes the proof.
\end{proof}

The partial order on $\mathcal{G}$ is defined by declaring $y\preceq x$ if and only if there exists a positive domain in $\mathcal{D}(x,y)$.  It is clear that the elements in different Alexander gradings are not comparable. The covering relations are indexed by the elements of $\mathcal{R}(x,y)$. It is routine to prove the following.

\begin{lem}\cite[Section 2.2]{CMPOZSzDT}\label{lem:gt}
With the grading assignment being the Maslov grading,  the grid poset $\mathcal{G}$ is a \hyperref[gtposet]{GT poset}.
\end{lem}

Following \hyperref[defn:chaingt]{Definition \ref*{defn:chaingt}}, let $C(\mc{G})$ be the chain complex associated to the GT poset $\mc{G}$. Its homology is bigraded, with the Maslov grading being the homological grading, and the Alexander grading being an extra grading.

\begin{thm} \cite{CMPOSS}The homology of $C(\mathcal{G})$ is   isomorphic, as bigraded $\F_2$-modules, to   $\mathit{HFK}^{-}(S^3,K;\F_2)$, the minus version of knot Floer   homology over $\F_2$.
\end{thm}

\section{Shellability}\label{sec:gss}

Let $G$ be a grid diagram of index $n$ drawn on a torus $T$, which represents a knot $K$. Let $\mc{G}$ be the associated GT poset.  We will show that each closed interval in $\mc{G}$ is EL-shellable.

Draw a circle $l$ which is disjoint from all the $\beta$ circles and intersects each $\alpha$ circle exactly once.
To an empty rectangle $R\in\mathcal{R}(x,y)$, we associate a triple $(s(R),i(R),t(R))$ in the following way: $s(R)$ is $0$ if $R$ intersects $l$ and is $1$ otherwise; if $s(R)=0$, $i(R)$ is the minimum number of $\beta$ circles that we have to cross to travel from $l$ to the leftmost arc of $R$, going left throughout; if $s(R)=1$, $i(R)$ is the minimum number of $\beta$ circles that we have to cross to go from $l$ to the leftmost arc of $R$, going right throughout; in both the cases, while counting the number of intersections, we include the leftmost $\beta$ arc of $R$; the number $t(R)$ always denotes the thickness of the empty rectangle $R$, which is the number of vertical annuli that $R$ hits.  The set of such triples is ordered lexicographically, and thus we have a map from the set of all covering relations to a totally ordered set. \hyperref[fig:trefoilrect]{Figure \ref*{fig:trefoilrect}} shows the grid diagram from \hyperref[fig:trefoil]{Figure \ref*{fig:trefoil}}, along with three generators $\wt{x}$, $\wt{y}$ and $\wt{z}$, represented by the white squares, the white circles and the black circles, respectively; the line $l$ is the dotted line; two empty rectangles $R_1\in\mc{R}(\wt{x},\wt{y})$ and $R_2\in\mc{R}(\wt{x},\wt{z})$ are shown; $(s(R_1),i(R_1),t(R_1))=(0,2,2)$ and $(s(R_2),i(R_2),t(R_2))=(1,2,1)$.  

\begin{figure} 
\psfrag{x1}{$X_1$}
\psfrag{x2}{$X_2$}
\psfrag{x3}{$X_3$}
\psfrag{x4}{$X_4$}
\psfrag{x5}{$X_5$}
\psfrag{o1}{$O_1$}
\psfrag{o2}{$O_2$}
\psfrag{o3}{$O_3$}
\psfrag{o4}{$O_4$}
\psfrag{o5}{$O_5$}
\psfrag{r1}{$R_1$}
\psfrag{r2}{$R_2$}
\psfrag{l}{$l$}
\begin{center}
\includegraphics[width=170pt]{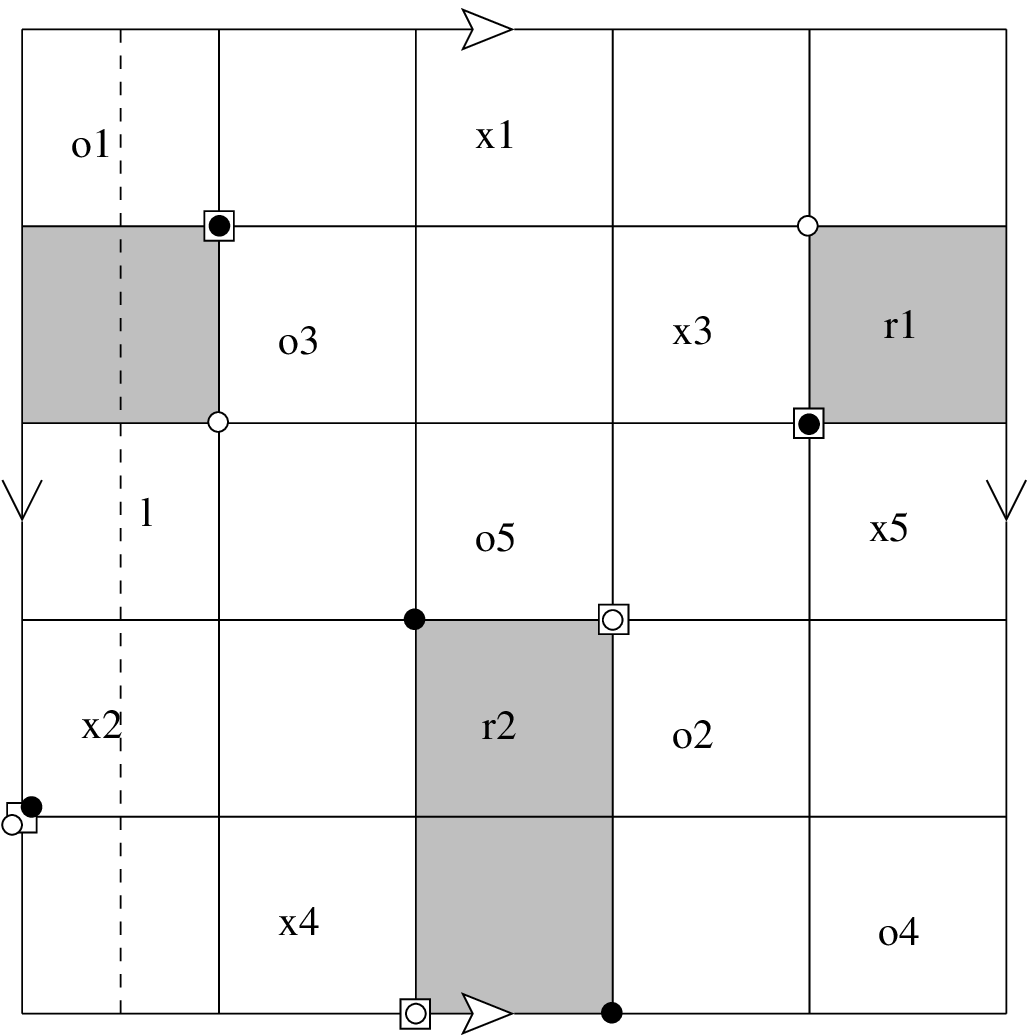}
\end{center}
\caption{Empty rectangles $R_1\in\mc{R}(\wt{x},\wt{y})$ and $R_2\in\mc{R}(\wt{x},\wt{z})$.}\label{fig:trefoilrect}
\end{figure}

\begin{thm}\label{thm:main}
 Let $x,y\in\mathcal{G}$. The map which sends a covering   relation represented by an \hyperref[emptyrectangle]{empty rectangle} $R$ to $(s(R),i(R),t(R))$   induces an \hyperref[el]{EL-shelling} on the closed interval $[y,x]$.
\end{thm}

Note that the interval $[y,x]$ is non-empty if and only if there is a positive domain in $\mathcal{D}(x,y)$.  From now on, we only consider that case. Recall from \hyperref[sec:posets]{Section   \ref*{sec:posets}} that given a maximal chain $\mf{m}=\{y\leftarrow z_1\leftarrow\cdots\leftarrow z_m\leftarrow x\}$ in $[y,x]$, we associate to it the labeling $((s,i,t)(y\leftarrow z_1),\ldots, (s,i,t)(z_m\leftarrow x))$, where $(s,i,t)(p\leftarrow q)$ is the $(s,i,t)$-triple associated to the empty rectangle corresponding to the covering relation $p\leftarrow q$. Also note that given $z\in\mathcal{G}$ and a triple $(s,i,t)$, there is at most one element $z'\in\mc{G}$ covering $z$, such that the covering relation corresponds to that triple.  Thus, no two maximal chains in $[y,x]$ have the same labeling. Therefore, there is a unique maximal chain $\mf{m}_0$ for which the labeling is lexicographically the minimum. The following two lemmas prove the above theorem.

\begin{lem}\label{lem:secondary}
 The lexicographically minimum labeling is an increasing labeling.
\end{lem}

\begin{proof} Assume not. Let $\mf{m}_0$ be the unique maximal chain   whose labeling is lexicographically the minimum. Let $p_1\leftarrow   p_2\leftarrow p_3$ be the first place in $\mf{m}_0$ where the   labeling decreases. Let $R_1$ and $R_2$ be the two empty rectangles   corresponding to the two covering relations.  Since each vertical   annulus and each horizontal annulus has at least one $X$ marking,   $\partial (R_1+R_2)$ is non-zero on on exactly three or exactly four   $\beta$ circles.

\begin{figure} 
\begin{center}
\includegraphics[width=330pt]{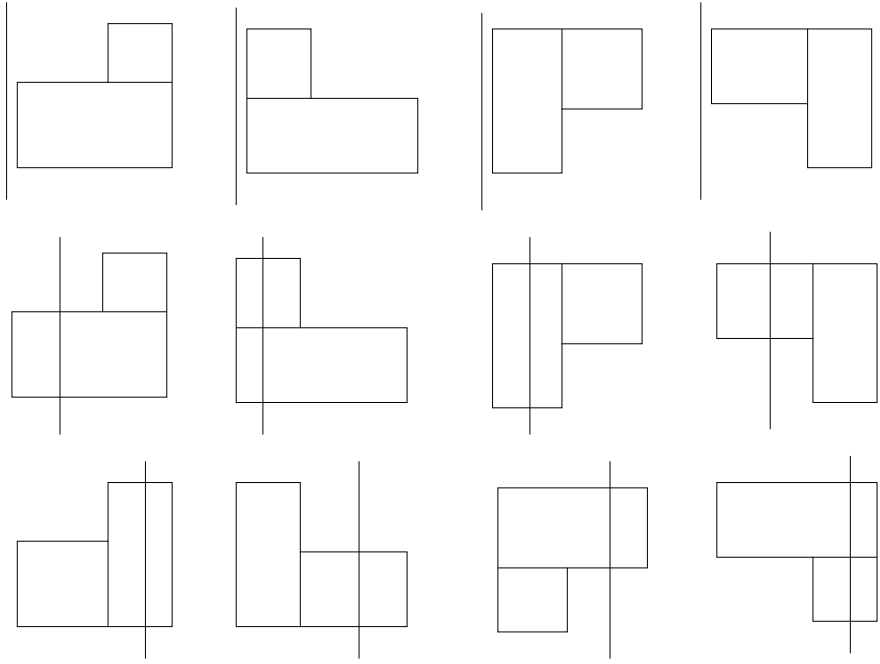}
\end{center}
\caption{The lexicographically smallest way to cut a hexagon.}\label{fig:cuthexagon}
\end{figure}

If $\del(R_1+R_2)$ is non-zero on exactly four $\beta$ circles, then switch $R_1$ and $R_2$, thereby producing a new maximal chain whose labeling is smaller than the labeling for $\mf{m}_0$ and thus contradicting the assumption that the labeling for $\mf{m}_0$ was the minimum. If, on the other hand, $\partial (R_1+R_2)$ is non-zero on exactly three $\beta$ circles, then $R_1+R_2$ looks like a hexagon. Depending on the shape of the hexagon and the position of the line $l$, only the cases as shown in \hyperref[fig:cuthexagon]{Figure   \ref*{fig:cuthexagon}} can occur. In each of the cases, the lexicographically smallest way to divide the hexagon is shown, and in each case, that happens to correspond to a chain where the labeling is increasing.  This proves that the labeling for the maximal chain $\mf{m}_0$ is increasing.
\end{proof}

\begin{lem} \label{lem:mainproof}
If the labeling corresponding to a maximal chain is an increasing labeling, then that maximal chain is the one whose labeling is lexicographically the minimum.
\end{lem}

\begin{proof} 
  Let $\mf{m}$ be a maximal chain whose labeling is increasing, and   let $\mf{m}_0$ be the unique maximal chain whose labeling is   lexicographically the minimum. We want to show that   $\mf{m}=\mf{m}_0$.

Starting at $y$, let us assume that $\mf{m}_1$ and $\mf{m}_2$ agree up to an element $z\in\mc{G}$.  Let $D$ be the unique positive domain in $\mathcal{D}(x,z)$. Let $\mf{m}_1=\mf{m}_0\cap[z,x]$ and let $\mf{m}_2=\mf{m}\cap[z,x]$. Let $R$ and $R'$ be the empty rectangles corresponding to the two covering relations on $z$ coming from the two chains $\mf{m}_1$ and $\mf{m}_2$.  We will show that $(s(R),i(R),t(R))=(s(R'),i(R'),t(R'))$ which would imply that $R=R'$; that, in turn would imply that $\mf{m}$ and $\mf{m}_0$ agree for at least one more generator, thus concluding the proof.

    Now, if $D$ does not intersect $l$, then $s$ is forced to be     $1$. On the other hand, if $D$ does intersect $l$, then eventually     in both $\mf{m}_1$ and $\mf{m}_2$, some covering relation will     have $s=0$, and since the labelings in both $\mf{m}_1$ and     $\mf{m}_2$ are increasing, they both must start with     $s=0$. Therefore, we see that $s$ is fixed.

    First we analyze the case when $s=1$. So assume that the whole     domain $D$ lies to the right of $l$, and let $i_0$ be the minimum     number of $\beta$ circles we have to cross to reach $D$ from $l$     going right throughout.  Clearly $i$, the second coordinate in the     triple $(s,i,t)$, can never be smaller than $i_0$. Furthermore,     since the whole domain $D$ has to be used up in both the chains     $\mf{m}_1$ and $\mf{m}_2$, so at some point, $i$ will be equal to     $i_0$. Since the labelings in both $\mf{m}_1$ and $\mf{m}_2$ are     increasing, this fixes $i=i_0$.

To see that $t$ is also fixed, we need an induction statement. Look at all $p$ of the form $z\leftarrow p\preceq x$, such that the covering relation $z\leftarrow p$ is by an empty rectangle with $i=i_0$. Let $R_0$ be the thinnest empty rectangle among them and let $t_0$ be the thickness of $R_0$. Our induction claim states: in any maximal chain in $[z,x]$, at some point we have to use an empty rectangle with $i=i_0$ and $t\leq t_0$. The induction is done on the length of the interval $[z,x]$.  Clearly when this length is $2$, the statement is true.  Let us assume that we do not start with the thinnest empty rectangle, but rather start with an empty rectangle $R_1$. Neither $R_0$, nor $R_1$ contains any coordinate of $z$ in its interior, and hence the local diagram must look like \hyperref[fig:thinrect1]{Figure \ref*{fig:thinrect1}}. Since the Maslov index of $D\setminus R_1$ is one lower than that of $D$, and since it has a starting empty rectangle with $(i,t)=(i_0,t_0)$, induction applies finishing the proof.

\begin{figure}
\psfrag{l}{$l$}
\psfrag{i0}{$i_0$}
\psfrag{t0}{$t_0$}
\psfrag{r0}{$R_0$}
\psfrag{tr0}{$R_1$}
\psfrag{z1}{$z_1$}
\psfrag{z2}{$z_2$}
\psfrag{z3}{$z_3$}
\begin{center} \includegraphics[width=170pt]{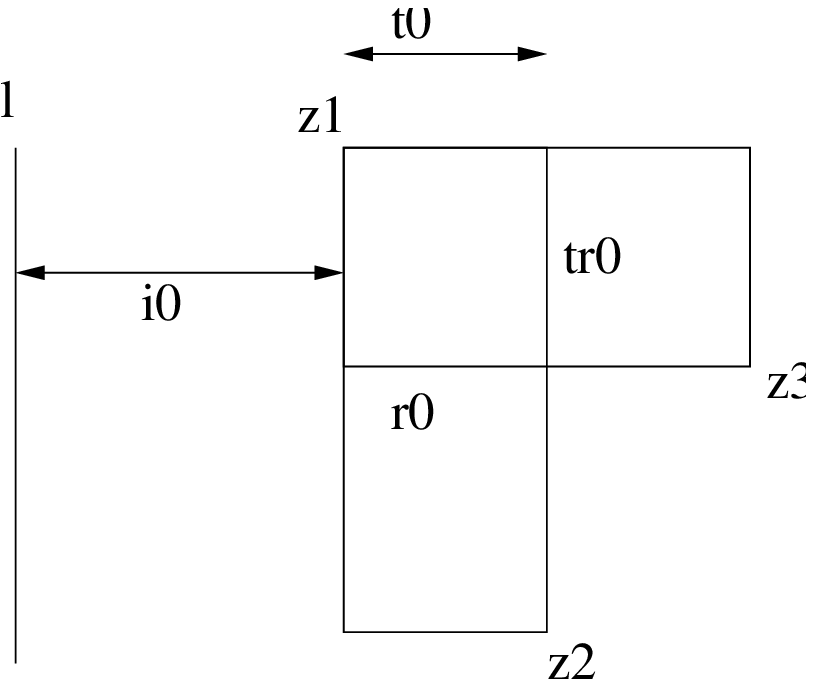}
\end{center}
\caption{Fixing the thickness of the starting empty rectangle when $s=1$.}\label{fig:thinrect1}
\end{figure}

 Thus in both the chains $\mf{m}_1$ and $\mf{m}_2$, at some point we have to use an empty rectangle with $i=i_0$ and $t\leq t_0$.  But since the labelings in both $\mf{m}_1$ and $\mf{m}_2$ are increasing, and $(i_0,t_0)$ is the smallest value of $(i,t)$ that we can start with, we have to start with $t=t_0$. Thus, this fixes $t$.

Now, let us assume that $s=0$. We need an induction statement to show that $i$ is fixed.  For each coordinate $z_i$ of $z$, consider the horizontal line segment $h_i$ lying on some $\alpha$ curve, which starts at $z_i$ and ends at $l$ and goes right throughout. We call $z_i$ admissible if every point just below the line segment $h_i$ belongs to $D$.  Since the starting empty rectangles in the chains $\mf{m}_1$ and $\mf{m}_2$ have $s=0$, there is at least one admissible coordinate. Among all the admissible coordinates, let $z_1$ be the one with $h_i$ having the smallest length. Let $i_0$ be the smallest length, measured by number of intersections with $\beta$ curves. Our induction claim states: in any maximal chain in $[z,x]$, at some point we have to use an empty rectangle with $s=0$ and $i\leq i_0$.  The induction is done on the length of $[z,x]$. Clearly when the length is $2$, the claim is true.  Let us assume that we start with an empty rectangle $R_0$ with $s=0$ and $i>i_0$. Since $R_0$ has Maslov index one, so it cannot contain any $z$ coordinate in its interior, and it also cannot contain any horizontal annulus. Therefore, it is easy to see that $R_0$ has to be disjoint from $h_1$, and thus $D\setminus R_0$ has Maslov index one lower than $D$ and still intersects $l$ and has an admissible coordinate with $h=i_0$. Thus induction applies, and proves our claim. It is obvious that the starting empty rectangles in the chains $\mf{m}_1$ and $\mf{m}_2$ must have $s=0$ and $i\geq i_0$.  Since both have increasing labelings, we must start with empty rectangles with $(s,i)=(0,i_0)$.

Next, we want to show that $t$ is also fixed. This is also done by an induction very similar to the ones above. Consider all $p$ with $z\leftarrow p\preceq x$, such that the covering relation $z\leftarrow p$ has $(s,i)=(0,i_0)$. Let $R_0$ be the thinnest empty rectangle among all such covering relations, and let $t_0$ be the thickness of $R_0$.  The induction claim states: in any maximal chain in $[z,x]$, at some point we have to use an empty rectangle with $(s,i)=(0,i_0)$ and $t\leq t_0$, and the induction is done on the length of $[z,x]$. Once again, it is trivial when the length is $2$. Assume that we start with a empty rectangle $R_1$ with $(s,i)=(0,i_0)$ and $t>t_0$.  The empty rectangles $R_0$ and $R_1$ must look like \hyperref[fig:thinrect0]{Figure \ref*{fig:thinrect0}}.

\begin{figure}
\psfrag{l}{$l$}
\psfrag{i0}{$i_0$}
\psfrag{t0}{$t_0$}
\psfrag{r0}{$R_0$}
\psfrag{tr0}{$R_1$}
\psfrag{z1}{$z_1$}
\psfrag{z2}{$z_2$}
\psfrag{z3}{$z_3$}
\begin{center} \includegraphics[width=170pt]{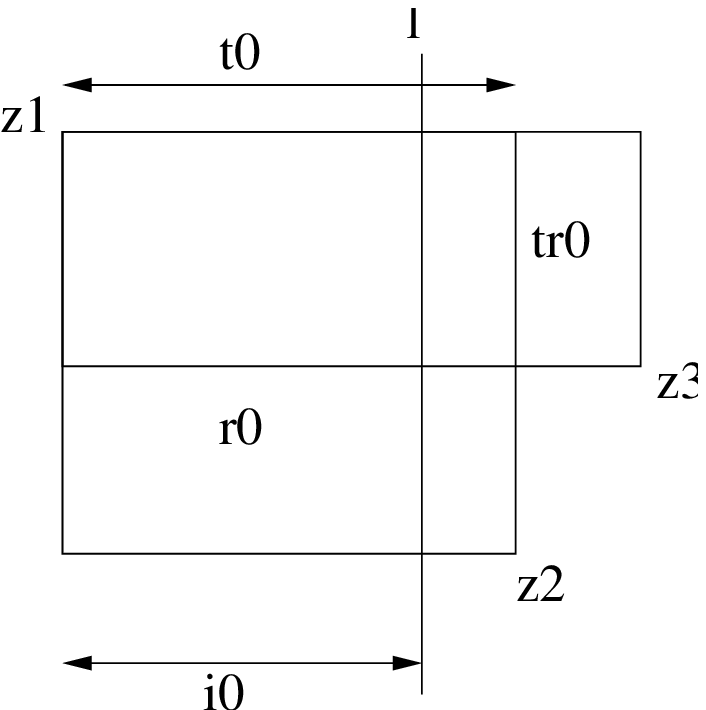}
\end{center}
\caption{Fixing the thickness of the starting empty rectangle when $s=0$.}\label{fig:thinrect0}
\end{figure}

Note that $D\setminus R_1$ has Maslov index one lower than $D$ and it still intersects $l$, and it still has an admissible coordinate with $h=i_0$. Thus induction applies. Since the labelings for $\mf{m}_1$ and $\mf{m}_2$ are both increasing, this implies that they both must start with an empty rectangle with $(s,i,t)=(0,i_0,t_0)$. Thus we see that the thickness is fixed.  As explained earlier, this finishes the proof.
\end{proof}

\begin{proof}[Proof of Theorem \ref*{thm:main}]
  Fix a closed interval $[y,x]$. By \hyperref[lem:mainproof]{Lemma \ref*{lem:mainproof}}, there is at most one maximal chain in $[y,x]$, namely the one whose labeling is lexicographically the minimum, for which the the labeling is increasing. \hyperref[lem:secondary]{Lemma \ref*{lem:secondary}} tells us that the lexicographically minimum labeling is also an increasing labeling. Therefore, there is a unique maximal chain in $[y,x]$ for which the labeling is increasing, and that labeling is lexicographically the minimum.
\end{proof}

\section{Flow category}

Inspired by the definition of the flow category in \cite{generalRCJJGS}, let us define a \emph{PL flow  category} $\mc{C}$ to be a small category with the following additional structures and properties: there is a grading assignment $g:\Ob_{\mc{C}}\rightarrow \Z$; for any $x\in\Ob_{\mc{C}}$, the set $\Mor_{\mc{C}}(x,x)$ consists of only the identity; for any two distinct $x,y\in\Ob_{\mc{C}}$, the set $\Mor_{\mc{C}}(x,y)$ is a (possibly empty) $(g(x)-g(y)-1)$-dimensional PL-manifold (possibly with boundary); for any $x,y,z\in\Ob_{\mc{C}}$, the map $\Mor_{\mc{C}}(x,y)\times\Mor_{\mc{C}}(y,z)\rightarrow\Mor_{\mc{C}}(x,z)$ is a PL-embedding, such that if $x,y,z$ are distinct, then the image of the embedding is a subspace of the boundary of $\Mor_{\mc{C}}(x,z)$; furthermore, if $p\in\Mor_{\mc{C}}(x,z)$ is a point on the boundary, then there exists $y\in\Ob_{\mc{C}}\sm\{x,z\}$, such that $p$ is in the image of the embedding $\Mor_{\mc{C}}(x,y)\times\Mor_{\mc{C}}(y,z)\hookrightarrow\Mor_{\mc{C}}(x,z)$.

In this section, given a GT poset $P$ whose every closed interval is shellable, we will construct a PL flow category $\mc{F}(P)$. There will also be a natural functor from $\mc{F}(P)$ to $\mc{C}(P)$, the category associated to the poset $P$.

The objects of $\mc{F}(P)$ are the elements of $P$, and the grading assignment on $\Ob_{\mc{F}(P)}$ is simply the grading assignment of $P$. The space $\Mor_{\mc{F}(P)}(x,y)$ is non-empty if and only if $y\preceq x$. If $y\prec x$, consider the interval $[y,x]$; let $B([y,x])$ be its barycentric subdivision, and now consider the interval $[\{y,x\},\infty)$ in $B([y,x])$; the space $\Mor_{\mc{F}(P)}(x,y)$ is defined to be its order complex, $X([\{y,x\},\infty))$. Let $z\prec y\prec x$; before we define the structure map $\Mor_{\mc{F}(P)}(x,y)\times\Mor_{\mc{F}(P)}(y,z)\rightarrow \Mor_{\mc{F}(P)}(x,z)$, observe that the Cartesian product of the poset $[\{y,x\},\infty)$, viewed as an interval in $B([y,x])$, and the poset $[\{z,y\},\infty)$, viewed as an interval in $B([z,y])$, is naturally isomorphic to the poset $[\{z,y,x\},\infty)$, viewed as an interval in $B([z,x])$; the structure map is the composition $X([\{y,x\},\infty))\times X([\{z,y\},\infty))=X([\{y,x\},\infty)\times[\{z,y\},\infty))= X([\{z,y,x\},\infty))\hookrightarrow X([\{z,x\},\infty))$.

\begin{thm}
For any \hyperref[gtposet]{GT poset} $P$ whose every closed interval is \hyperref[shellable]{shellable}, the category $\mc{F}(P)$ is a PL flow category. Furtheremore, whenever $y\prec x$, the space $\Mor_{\mc{F}(P)}(x,y)$ is PL-homeomorphic to a ball.
\end{thm}

\begin{proof}
  Let $y\prec x$. We will prove that   $\Mor_{\mc{F}(P)}(x,y)=X([\{y,x\},\infty))$ is PL-homeomorphic to   the $(g(x)-g(y)-1)$-dimensional ball. Since the interval $[y,x]$ is   graded with length $g(x)-g(y)+1$, the interval $[\{y,x\},\infty)$ in   $B([y,x])$ is graded with length $g(x)-g(y)$.  Since $P$ is locally   thin, the interval $(y,x)$ is thin, and therefore, the interval   $[\{y,x\},\infty)$ in $B([y,x])$ is subthin; furthermore, a   submaximal chain of $[\{y,x\},\infty)$ is covered by exactly one   maximal chain if and only if it does not contain $\{y,x\}$. Finally,   since the interval $[y,x]$ is shellable, we know from   \hyperref[thm:shellbasic]{Theorem \ref*{thm:shellbasic}} that the   interval $[\{y,x\},\infty)$ is shellable. Therefore,   \hyperref[thm:ordercomplex]{Theorem \ref*{thm:ordercomplex}} applies   and tells us that the order complex $X([\{y,x\},\infty))$ is   PL-homeomorphic to the $(g(x)-g(y)-1)$-dimensional ball;   furthermore, the boundary of the ball corresponds precisely to the   submaximal chains of $[\{y,x\},\infty)$ that do not contain   $\{y,x\}$.

Now let $z\prec y\prec x$. Next, we will prove that the map $\Mor_{\mc{F}(P)}(x,y)\times\Mor_{\mc{F}(P)}(y,z)\rightarrow \Mor_{\mc{F}(P)}(x,z)$ is a PL-embedding into the boundary of $\Mor_{\mc{F}(P)}(x,z)$. The map is an embedding because it is essentially the inclusion of $X([\{z,y,x\},\infty))$ as a subcomplex of $X([\{z,x\},\infty))$. Its image lies in the boundary of $\Mor_{\mc{F}(P)}(x,z)$ because none of the chains in $[\{z,y,x\},\infty)$ contain the point $\{z,x\}$.

Finally, let us prove that every point in the boundary of $\Mor_{\mc{F}(P)}(x,z)$ is in the image of such an embedding. Let $p$ be a point in the boundary. Let $\Delta$ be a maximal dimensional simplex in the boundary of $X([\{z,x\},\infty))$ that contains $p$. Let $C$ be the submaximal chain in $[\{z,x\},\infty)$ that corresponds to $\Delta$. Since $\Delta$ lies in the boundary, $C$ does not contain the element $\{z,x\}$. Therefore, the smallest element of $C$ is some element of the form $\{z,y,x\}$, with $z\prec y\prec x$. Then $p$ lies in the image of the map $\Mor_{\mc{F}(P)}(x,y)\times\Mor_{\mc{F}(P)}(y,z)\rightarrow \Mor_{\mc{F}(P)}(x,z)$.
\end{proof}

Therefore, using \hyperref[lem:gt]{Lemma \ref*{lem:gt}}, \hyperref[thm:main]{Theorem \ref*{thm:main}} and \hyperref[thm:el]{Theorem \ref*{thm:el}}, we can associate a PL flow category to a grid diagram in a natural way. This suggests that we might be able to associate a stable homotopy type to a grid diagram in a natural way, whose homology will be the knot Floer homology. However, in order to apply the Cohen-Jones-Segal machinery \cite{generalRCJJGS}, we need to frame the tangent bundles of $\Mor_{\mc{F}(P)}$ in a coherent way: namely, for every $y\prec x$, we want a trivialization of the bundle $T_*(\Mor_{\mc{F}(P)}(x,y))\oplus\R$ over $\Mor_{\mc{F}(P)}(x,y)$, such that whenever $z\prec y\prec x$, the trivialization of the bundle $T_*(\Mor_{\mc{F}(P)}(x,y))\oplus\R\oplus T_*(\Mor_{\mc{F}(P)}(y,z))\oplus\R$ over $\Mor_{\mc{F}(P)}(x,y)\times\Mor_{\mc{F}(P)}(y,z)$ induces the trivialization of the bundle $T_*(\Mor_{\mc{F}(P)}(x,z))\oplus\R$ over the image of the inclusion $\Mor_{\mc{F}(P)}(x,y)\times\Mor_{\mc{F}(P)}(y,z)\hookrightarrow\Mor_{\mc{F}(P)}(x,z)$.

It is not clear what is the relevant notion of the tangent bundle of a PL-manifold. More importantly, it is not clear how to produce these coherent framings starting with a GT poset whose every closed interval is shellable. It seems likely that to do so we need more structures on the poset. It would be an interesting endeavor to characterize those extra structures, and to check if the GT posets arising from grid diagrams have them.

\bibliography{homotopynew}

\end{document}